\documentclass[12pt]{amsart}
\usepackage{amsmath,amssymb,latexsym,amsthm,amscd,stmaryrd,amsfonts,mathrsfs}
\usepackage[title,titletoc,toc]{appendix}
\usepackage[all]{xy}
\usepackage{soul}

\theoremstyle{definition}
\newtheorem{Def}[subsubsection]{Definition}
\theoremstyle{plain}
\newtheorem{prop}[subsubsection]{Proposition}
\newtheorem{thm}[subsubsection]{Theorem}
\newtheorem{lem}[subsubsection]{Lemma}

\newcommand{\mbf}{\mathbf}
\newcommand{\mfk}{\mathfrak}
\newcommand{\mscr}{\mathscr}
\newcommand{\mcal}{\mathcal}
\newcommand{\mbb}{\mathbb}
\newcommand{\mrm}{\mathrm}
\newcommand{\A}{\mathfrak a}

\newcommand{\ro}{\mrm{ro}}
\newcommand{\co}{\mrm{co}}

\newcommand{\wt}{\widetilde}

\newcommand{\U}{\mbf U}

\usepackage{color}
\newcommand{\nc}{\newcommand}
\nc{\redtext}[1]{\textcolor{red}{#1}}
\nc{\bluetext}[1]{\textcolor{blue}{#1}}
\nc{\greentext}[1]{\textcolor{green}{#1}}
\nc{\yl}[1]{\redtext{From yq: #1}}
\nc{\zb}[1]{\redtext{From zb: #1}}

\title[Geometric RTT relation of $U_v(gl_n)^+$]{Geometric RTT realization of $U_v(gl_n)^+$}

\author{Haitao Ma}
\author{Ming Liu}
\author{Zhu-Jun Zheng}
\address{Department of Mathematics\\ South China University of Technology, Wushan Rd, Guangzhou, China 510640}
\email{Zhengzj@scut.edu.cn}

\date{}

\begin{document}

\begin{abstract}
In this paper, we give a  BLM¡¡ realization of the positive part of the quantum group of $U_v(gl_n)$ with respect to RTT¡¡ relations.

\end{abstract}

\maketitle

\section{Introduction}

There are two ways to study quantum groups. One is algebraic, the another is geometric.

In algebriaic way, there are  two  methods.  The first was adopted by
Drinfeld \cite{D1,D3} and Jimbo \cite{J} to define the quantum enveloping algebra $U_q(\mathfrak{g})$ as a $q$-deformation of the
enveloping algebra $U(\mathfrak g)$
in terms of the Chevalley generators and Serre relations based on the data coming from the corresponding Cartan matrix.
The second approach to realize the quantum groups was through RTT ~¡¡method. The approach was  from the quantum inverse scattering method developed by the Leningrad school. Sometimes it was more available for us. For example, In \cite{FRT} Faddeev, Reshetikhin and Takhtajan have shown that both the quantum enveloping algebras $U_q(\mathfrak g)$ and the
dual quantum groups for finite classical simple Lie algebras $\mathfrak g$ can be studied in the RTT method using the solutions $R$ of the
Yang-Baxter equation:
\begin{equation}\label{YB Eq}
R_{12}R_{13}R_{23}=R_{23}R_{13}R_{12}.
\end{equation}

In geometric way, there are many results on  the geometric realization of quantum groups with respect to the Chevalley generators. For example, the $\mathbb{C}$-valued $GL_n$ equivalence functions on the flag variety associated to the finite dimensional vector space over $\mathbb{F}_q$ was considered by Beilinson,  Lusztig,  and McPherson.
It  gave realization of Schur algebra of $U_v(gl_n)$, and $U_v(gl_n)$ was the limit algebra of the Schur algebra. Also Du and Gu used the same way give the realization of $U_v(gl(m|n))$\cite{DG14}. Moreover, they also obtained
a new basis containing the standard generators for quantum linear super group and explicit multiplication
formulas between the generators and an arbitrary basis element. In affine case, Fu gave
a BLM realization for $ U_{\mbb Z} (\widehat{gl}_n)$\cite{F12}. By the similar way,  Bao-Wang \cite{BKLW13} and Fan-Li\cite{FL14} gave the several new quantum groups
and the Schur-like duality of type B/C and D.  But there are few results on the geometric RTT  realization. So the questions are if  the the geometric realization of quantum groups with respect to  the RTT relations can be given and if more information about the representation of  quantum group can be given though the geometry RTT realization. In this paper, we use the BLM's way to construct the RTT realization of $U_v(gl_n)^{+}$.

The paper is organized as follows. In section 2,
we recall the basic results on flag varieties and the RTT realization of $U_{v}(gl_n)$.
In section 3, we calculate generator realizations of the Schur algebra,  find a subalgebra of the limit algebra of the Schur algebra and show that it is isomorphic to the positive part of $U_{v}(gl_n)$.

\section{Preliminary}
In this section, let us recall some basic facts on flag varieties appear in \cite{BLM90} and the  RTT  realization of $U_v(gl_n)$.

Let $\mbb F_q$ be a finite field of $q$ elements and of odd characteristic.
For positive integers  $d$ and $n$,
 consider the set $\mscr X$ of $n$-step flags $V=( V_i)_{ 0\leq i\leq n}$
          in $\mbb F_q^d$ such that  $V_0 = 0$, $V_i\subseteq V_{i+1}$.

Let $G=\mrm{GL}(\mbb F_q^d)$,
and $G$ acts naturally on set $\mscr X$.
Let $G$ act diagonally on the product $\mscr X\times \mscr X$.
Set
\begin{equation}
  \mcal A=\mbb Z[v^{\pm1}].
\end{equation}
Let
$
\mcal S_{\mscr X}=\mcal A(\mscr X\times\mscr X)^G
$
be the set of all $\mcal A$-valued $G$-invariant functions on $\mscr X\times \mscr X$.
Clearly, the set $\mcal S_{\mscr X}$ is a  free $\mcal A$-module.
Moreover,  $\mcal S_{\mscr X}$ admits  an associative  $\mcal A$-algebra structure `$*$' under a standard  convolution product as discussed in ~\cite{BLM90}. In particular, when $v$ is specialized to $\sqrt q$, we have
\begin{equation}
  \label{eq30}
  f * g(V, V')=\sum_{V''\in \mscr X}f(V, V'')g(V'',V'), \quad \forall\ V,V'\in \mscr X.
\end{equation}

Let us describe the $G$-orbits on  $\mscr X\times \mscr X$.
We start by introducing the following notations associated to a matrix
$M=(m_{ij})_{1\leq i, j \leq c}$.
\begin{align} \label{ro-co}
\begin{split}
\ro (M) & = \left (\sum_{j=1}^{c}  m_{ij} \right )_{1\leq i\leq c}, \\
\co (M) & = \left (\sum_{i=1}^{c} m_{ij} \right )_{1\leq j \leq c}.
\end{split}
\end{align}
We also write $\ro(M)_i$ and $\co (M)_j$ for the $i$-th and $j$-th component of the row vectors of $\ro (M)$ and $\co(M)$, respectively.
For any pair $(V, V')$ of flags in $\mscr X$, we can assign an $n$ by $n$ matrix whose $(i,j)$-entry equal to
$\dim \frac{V_{i-1}+ V_i\cap V_j'}{V_{i-1} + V_i\cap V_{j-1}'}$. We have the following bijection.

\begin{align}
G \backslash \mscr X\times \mscr X \simeq  \Theta_d,
\end{align}
where
$\Theta_d$ is the set of all matrices $\Theta_d$ in $\mbox{Mat}_{n \times n} (\mbb N)$ such that
$\sum_{i,j}(\Theta_d)_{i,j} = d$

\begin{Def}
The algebra $U_v(gl_n)$ is generated by the elements $t_{ij}$ and $\overline{t_{ij}}$ with $1 \leq i, j \leq n$ subject to the relations

\begin{eqnarray*}
  &&t_{ij} = \overline{t_{ij}} = 0, \ 1 \leq i < j \leq n.\\
 &&v^{ \delta_{ij}} t_{ia}t_{jb} - v^{  \delta_{ab}} t_{jb}t_{ia} = (v - v^{-1})(\delta_{b < a} - \delta_{i < j})t_{ja}t_{ib},\\
  &&v^{\delta_{ij}} \overline{t}_{ia}\overline{t}_{jb} - v^{\delta_{ab}} \overline{t}_{jb}\overline{t}_{ia} = (v- v^{-1})(\delta_{b < a} - \delta_{i < j})\overline{t}_{ja}\overline{t}_{ib},\\
 &&v^{\delta_{ij}} \overline{t}_{ia}t_{jb} - v^{\delta_{ab}} t_{jb}\overline{t}_{ia} = (v - v^{-1})(\delta_{b < a}t_{ja}\overline{t}_{ib} - \delta_{i < j}\overline{t}_{ja}t_{ib}).
\end{eqnarray*}
\end{Def}

\section{BLM Realization Of $U_{v}(gl_n)^{+}$ With Respect To RTT Relations}
\subsection{Calculus of the algebra $\mcal S$}

For simplicity, we shall denote $\mcal S$ instead of $\mcal S_{\mscr X}$.
In this section, we determine  the generators for $\mcal S$ and the associated  multiplication formula.

\begin{lem}
Assume $V_1,~ V_2,~ V_3$ are the vector space over $F_q$, $V_1\overset{1}{\subset} V_2 \subset V_3$,~and $\mathrm{dim}V_1 = n - 1$, $\mathrm{dim}V_2 = n$, $\mathrm{dim}V_3 = m$.  Set $S = \{V_3' | V_3'\overset{1}{\subset} V_3,V_1 \subset V_3',  V_2 \cap V_3' \neq V_2 \}$. Then
$$\sharp S = q^{m - n}.$$
Where $V_1\overset{1}{\subset} V_2$ means $V_1\subset V_2$ and $\mathrm{dim}V_2 - \mathrm{dim}V_1 = 1.$
\end{lem}
\begin{proof}
$\sharp S = \sharp\{V_{3}' | V_{3}' \overset{1}{\subset} V_3, V_1 \subset V_3'\} - \sharp\{V_{3}' | V_{3}' \overset{1}{\subset} V_3, V_2 \subset V_3'\}$

$
\begin{array}{ccc}
  \sharp S & = & \sharp\{V_{3}' | V_{3}' \overset{1}{\subset} V_3, V_1 \subset V_3'\} - \sharp\{V_{3}' | V_{3}' \overset{1}{\subset} V_3, V_2 \subset V_3'\} \\
   & = & \frac{q^{m - n +1} - 1}{q - 1} - \frac{q^{m - n} - 1}{q - 1}\\
   & = & q^{m - n}
\end{array}
$

\end{proof}

\begin{lem}\label{lem-x}
Assume $V_1, ~V_2,~ V_3$ are the vector space over $F_q$, $V_2\subset V_3$, $V_1\subset V_3$, $V_1 \cap V_2 = 0$, $\mathrm{dim}V_1  = 1$, $\mathrm{dim}V_2 = n$, $\mathrm{dim} V_3 = m$.  Set $S = \{V_3' | V_3'\overset{1}{\subset} V_3,V_1 \cap V_3' = 0,$ $ \mathrm{dim} V_2 \cap V_3'= n - 1 \}$. Then
$$\sharp S = q^{m - 1} - q^{m - n - 1}.$$
\end{lem}
\begin{proof}
$\sharp S = \sharp\{V' | V' \overset{1}{\subset} V_3\} - \sharp\{V' | V' \overset{1}{\subset} V_3, V_1 \subset V'\} - \sharp\{V' | V' \overset{1}{\subset} V_3, V_2 \subset V'\} + \sharp\{V' | V' \overset{1}{\subset} V_3, V_1 \subset V', V_2 \subset V'\}$

$= \frac{q^m - 1}{q - 1} - \frac{q^{m - 1} - 1}{q - 1} - \frac{q^{m - n} - 1}{q - 1} + \frac{q^{m - n - 1} - 1}{q - 1} $

$= q^{m - 1} - q^{m - n - 1}.$

\end{proof}

\begin{lem} \label{lem-y}
Let $A = (a_{ij}) \in \Theta_d$.

(a)Assume $B = (b_{ij}) \in \Theta_d$. There exist $1\leq i_0 < i_1  \leq n$  such that  $B - E_{i_{0},i_{1}}$ is the diagonal matrices, and $\sum_i b_{ij} = \sum_k a_{jk}.$ Then
$$e_B \ast e_A = \sum_{\textbf{(j,p)}}f_{\textbf{(j,p)}}e_{(A + \sum_{l = 1}^{m}(E_{j_{l - 1},p_{l}} - E_{j_l,p_l}))},$$
where $\textbf{(j,p)} = ((j_1,p_1),\cdots,(j_m,p_m))$ satisfied the conditions:
$i_0 = j_0 <j_1 < \cdots <j_m = i_1$, $1 \leq p_m < \cdots <p_1 \leq n$, and for any $1 \leq k \leq m$, $a_{j_k,p_k} \geq 1$. $f_{\textbf{(j,p)}}= f_1f_2 \cdots f_m$, where

$ f_l =\left\{\begin{array}{ll}
\frac{v^{2(1 + \sum{j \geq p_1}a_{j_0,j})} - v^{2\sum_{j > p_1}a_{j_0,j}}} {v^2 - 1}\prod\limits_{k = j_0 + 1}^{k = j_1 - 1}v^{2\sum_{j\geq p_1}a_{k,j}}& \text{if} \ l = 1; \\[.15in]
(v^{2\sum_{j\geq p_l}a_{j_{l - 1},j}} - v^{2(\sum_{j> p_l}a_{j_{l - 1},j}- 1 )}) \prod\limits_{k = j_{l - 1} + 1}^{k = j_l - 1}v^{2\sum_{j\geq p_l}a_{k,j}}& \text{if}\  l>1.
 \end{array}   \right.
$

(b)Assume $C = (c_{ij}) \in \Theta_d$. There exist $n \geq i_0 > i_1  \geq 1$  such that  $C - E_{i_{0},i_{1}}$ is the diagonal matrices, and $\sum_i c_{ij} = \sum_k a_{jk}.$ Then
$$e_C \ast e_A = \sum_{\textbf{(j,p)}}f_{\textbf{(j,p)}}'e_{(A + \sum_{l = 1}^{m}(E_{j_{l - 1},p_{l}} - E_{j_l,p_l}))},$$
where $\textbf{(j,p)} = ((j_1,p_1),\cdots,(j_m,p_m))$ satisfied the conditions:
$i_0 = j_0 >j_1 > \cdots >j_m = i_1$, $1 \leq p_1 < \cdots <p_m \leq n$, and for any $1 \leq k \leq m$, $a_{j_k,p_k} \geq 1$. $f_{\textbf{(j,p)}}'= f_1'f_2' \cdots f_m'$, where

$ f_l' =\left\{\begin{array}{ll}
\frac{v^{2(1 + \sum_{j \leq p_1}a_{j_0,j})} - v^{2\sum_{j < p_1}a_{j_0,j}}} {v^2 - 1}\prod\limits_{k = j_1 + 1}^{k = j_0 - 1}v^{2\sum_{j\leq p_1}a_{k,j}}& \text{if} \ l = 1; \\[.15in]
(v^{2\sum_{j\leq p_l}a_{j_{l - 1},j}} - v^{2(\sum_{j < p_l}a_{j_{l - 1},j}- 1 )}) \prod\limits_{k = j_{l} + 1}^{k = j_{l - 1} - 1}v^{2\sum_{j\leq p_l}a_{k,j}}& \text{if}\  l>1.
 \end{array}   \right.
$

\end{lem}

\begin{proof}
 Assume $A' = A + \sum_{l = 1}^{m}(E_{j_{l - 1},p_{l}} - E_{j_l,p_l})$. Let $f = (V_1 \subset \cdots \subset V_{i_0 - 1} \subset V_{i_0} \subset \cdots \subset V_{i_1 - 1} \subset V_{i_1} \subset \cdots \subset V_n),~ f' = (V_1' \subset \cdots \subset V_{i_0 - 1}' \subset V_{i_0}' \subset \cdots \subset V_{i_1 - 1}' \subset V_{i_1}' \subset \cdots \subset V_n')$ be such that $(f,f') \in \mathcal{O}_{A'}$. Set $V_0 = V_0' = 0.$

 To prove (a), we need to compute how many $f''$  such that $(f, ~f'') \in \mathcal{O}_B$, $(f'',~f') \in \mathcal{O}_A$. Assume $ f'' = (V_1'' \subset \cdots \subset V_{i_0 - 1}'' \subset V_{i_0}'' \subset \cdots \subset V_{i_1 - 1}'' \subset V_{i_1}'' \subset \cdots \subset V_n'').$
 For any $i < i_0$ or $i \geq i_1$, $V_i'' = V_i$.

 First, we need to count how many $V_{i_0}''$ there exist. Consider the set $Z_{i_0}$ of all subspace $U_{i_0}$ of $V$ such that $ V_{i_0 - 1}\subset U_{i_0} \overset{1}{\subset} V_{i_0}$, $V_{i_0 - 1} + V_{i_0} \cap V_{p_1 - 1}' \subset U_{i_0} $, and$ (V_{i_0 - 1} + V_{i_0} \cap V_{p_1}')\cap U_{i_0} \neq V_{i_0 - 1} + V_{i_0} \cap V_{p_1}'$.

 $\sharp Z_{i_0} = \sharp \{U_{i_0} | U_{i_0} \supset V_{i_0 - 1} + V_{i_0} \cap V_{p_1 - 1}' \} - \sharp \{U_{i_0} | U_{i_0} \supset V_{i_0 - 1} + V_{i_0} \cap V_{p_1}' \}$

 $ = \frac{v^{2(1 + \sum_{j \geq p_1}a_{j_0,j})} - v^{2\sum_{j > p_1}a_{j_0,j}}} {v^2 - 1}.$

Second, we fix $U_{i_0}$. Consider the set  $Z_{i_0 + 1}$ of all subspace $U_{i_0 + 1}$ of $V$ such that  $U_{i_0}\subset U_{i_0 + 1} \overset{1}{\subset} V_{i_0 + 1}$, $V_{i_0} \cap U_{i_{0} + 1} = U_{i_0}$(since $B_{i_{0},~i_{0} + 1} = 0$), $U_{i_0} + V_{i_0 + 1} \cap V_{p_1 - 1} \subset U_{i_0 + 1}$. That is, $V_{i_0 + 1} \cap V_{p_1 - 1}' + U_{i_0} \subset U_{i_0 + 1}, V_{i_0} \cap U_{i_{0} + 1} = U_{i_0}.$

$\sharp Z_{i_0 +1} =  \sharp \{U_{i_0 + 1} | U_{i_0 + 1} \supset V_{i_0 + 1} \cap V_{p_1 - 1}' + U_{i_0}  \} - \sharp \{U_{i_0 + 1} | U_{i_0 + 1} \supset V_{i_0} + V_{i_0 + 1} \cap V_{p_1 - 1}' \}$

$ = \frac{v^{2(1 + \sum_{j \geq p_1}a_{j_0 + 1,j})} - v^{2\sum_{j \geq p_1}a_{j_0 + 1,j}}} {v^2 - 1}$

$= v^{2\sum_{j \geq p_1}a_{j_0 + 1,j}}.$

Similarly, as above, for any $i_0 + 1 < k <j_1$,  $\sharp Z_{k} = v^{2\sum_{j\geq p_1}a_{i_k,j}}$. Then we get the coefficient of $f_1$.

Finally, we fix $U_{k},~ i_0 \leq k \leq j_1 -1.$ We need to count how many $V_{j_1}''$  exist. Consider the set $Z_{j_1}$ of all subspace $U_{j_1}$ of $V$ such that  $U_{j_1 - 1}\subset U_{j_1} \overset{1}{\subset} V_{j_1}$, $V_{i_0}\cap U_{j_1} = V_{i_0} \cap U_{j_1 - 1} = U_{i_0}$ (since $b_{j_0,j_1} = 0$), $U_{j_1 - 1} + V_{j-1} \cap V_{P_2 - 1}' \subset U_{j_1}$, $(U_{j_1 - 1} + V_{j-1} \cap V_{P_2}') \cap U_{j_1} \neq U_{j_1 - 1} + V_{j-1} \cap V_{P_2}'.$  By the lemma$\ref{lem-x}$. $\sharp Z_{j_1} = q^{\sum_{j\geq p_2}a_{j_1,j}} - q^{\sum_{j> p_2}a_{j_{1},j} - 1}$.

All other $Z_{k}\ (j_1+1 \leq k \leq i_1 -1)$ can be counted by the similar way as above. Then $(a)$ follows.

To prove (b), we need to compute how many $f''$  such that $(f,~f'') \in \mathcal{O}_B$, $(f'',~f') \in \mathcal{O}_A$. Assume $ f'' = (V_1'' \subset \cdots \subset V_{i_0 - 1}'' \subset V_{i_0}'' \subset \cdots \subset V_{i_1 - 1}'' \subset V_{i_1}'' \subset \cdots \subset V_n'').$
 For any $i \geq i_0$ or $i < i_1$, $V_i'' = V_i$.

 First, we need to count how many $V_{i_0 - 1}''$ exist. Consider the set $Z_{i_0 - 1}$ of all subspace $U_{i_0 - 1}$ of $V$ such that $ V_{i_0 - 1}\overset{1}{\subset} U_{i_0 - 1} \subset V_{i_0}$, $(V_{i_0 - 1} + V_{i_0} \cap V_{p_1 - 1}') \cap U_{i_0 - 1} \neq  U_{i_0 - 1} $, and $ U_{i_0 - 1} \subset V_{i_0 - 1} + V_{i_0} \cap V_{p_1}'$.

$\sharp Z_{i_0 - 1} = \sharp \{U_{i_0 - 1} | U_{i_0 - 1} \subset V_{i_0 - 1} + V_{i_0} \cap V_{p_1 }' \} - \sharp \{U_{i_0} | U_{i_0} \subset V_{i_0 - 1} + V_{i_0} \cap V_{p_1 - 1}' \}$

 $ = \frac{v^{2(1 + \sum_{j \leq p_1}a_{j_0,j})} - v^{2\sum_{j < p_1}a_{j_0,j}}} {v^2 - 1}.$

Second, we fix $U_{i_0 - 1}$. Consider the set  $Z_{i_0 - 2}$ of all subspace $U_{i_0 - 2}$ of $V$ such that  $V_{i_0 - 2}\overset{1}{\subset} U_{i_0 - 2} \subset U_{i_0 - 1}$, $V_{i_0 - 1}  +  U_{i_{0} - 2} = U_{i_0 - 1}$(since $c_{i_{0},i_{0} - 1} = 0$), $ U_{i_0 - 2} \subset V_{i_0 - 2} + U_{i_0 - 1} \cap V_{p_1}' $. That is, $U_{i_0 - 2} \subset V_{i_0 - 2} + U_{i_0 - 1} \cap V_{p_1}',U_{i_0 - 2} \cap( V_{i_0 - 2} + V_{i_0 - 1} \cap V_{p_1}') \neq U_{i_0 - 2}.$

$\sharp Z_{i_0 - 2} =  \sharp \{U_{i_0 - 2} | U_{i_0 - 2} \subset V_{i_0 - 2} + U_{i_0 - 1} \cap V_{p_1}' \} - \sharp \{U_{i_0 - 2} | U_{i_0 - 2} \subset V_{i_0 - 2} + V_{i_0 - 1} \cap V_{p_1}'\}$

$ = \frac{v^{2(1 + \sum_{j \leq p_1}a_{j_0 - 1,j})} - v^{2\sum_{j \leq p_1}a_{j_0 - 1,j}}} {v^2 - 1}$

$= v^{2\sum_{j \leq p_1}a_{j_0 - 1,j}}.$

Similarly, as above, for any $j_1  \leq k <j_0 - 2$,  $\sharp Z_{k} = v^{2\sum_{j\leq p_1}a_{ k + 1,j}}$. Then we get the coefficient of $f_1$

Finally, we fix $U_{k},~ j_1 \leq k \leq j_0 - 1.$ We need to count how many $V_{j_1 - 1}''$  exist. Consider the set $Z_{j_1 - 1}$ of all subspace $U_{j_1 - 1}$ of $V$ such that  $V_{j_1 - 1}\overset{1}{\subset}U_{j_1 - 1}\subset U_{j_1}$, $U_{j_1} = V_{j_1} + U_{j_1 - 1}$, $U_{j_1 - 1} \subset V_{j_1 - 1} + U_{j_1} \cap V_{p_2}'$, $U_{j_1 - 1} \overset{\bullet}{\subset} V_{j_1 - 1} + U_{j_1} \cap V_{p_2 - 1}'$, where $U_{j_1 - 1} \overset{\bullet}{\subset} V_{j_1 - 1} + U_{j_1} \cap V_{p_2 - 1}'$ means $U_{j_1 - 1} \subset V_{j_1 - 1} + U_{j_1} \cap V_{p_2 - 1}'$ and $U_{j_1 - 1} \cap( V_{j_1 - 1} + V_{j_1} \cap V_{p_2 - 1}')\neq U_{j_1 - 1}$.

$\sharp Z_{j_1 - 1} = \frac{q^{(1 + \sum_{k \leq p_2})} - 1}{q - 1} - \frac{q^{\sum_{k < p_2}} - 1}{q - 1} - (\frac{q^{\sum_{k \leq p_2}} - 1}{q - 1} - \frac{q^{(\sum_{k < p_2} - 1)} - 1}{q - 1})$

$= q^{\sum_{k \leq p_2}} - q^{(\sum_{k < p_2} - 1)}.$

All other $Z_{k}\ (j_m  \leq k \leq j_1 - 2)$ can be counted by the similar way as above. Then $(b)$ follows.

\end{proof}

\begin{prop}
Let $A = (a_{ij}) \in \Theta_d$.
Assume $B = (b_{ij}) \in \Theta_d$. There exist $1\leq i_0 < i_1 < \cdots <i_{m - 1} < i_m \leq n$  such that  $B -\sum_{k =1}^{m} E_{i_{k-1},i_{k}}$ is the diagonal matrices, and $\sum_i b_{ij} = \sum_k a_{jk}.$ Then
$$e_B \ast e_A = \sum_{\textbf{(j,p)}} f_{\textbf{(j,p)}} e_{(A + \sum_{1 \leq k \leq m, 1 \leq l \leq r_k}(E_{j_{k,l - 1},p_{k, l}} - E_{j_{k.l},p_{k,l}}))},$$
where $\textbf{(j,p)}$ runs over $((\textbf{j}_{1},~\textbf{p}_{1}),~\cdots,~(\textbf{j}_{m},~\textbf{p}_{m}))$. For any $1 \leq k \leq m$,
$$(\textbf{j}_{k},~\textbf{p}_{k}) = ((j_{k,1},~p_{k_1}),~\cdots,~(j_{k,r_k},~p_{k,r_k}))$$
 satisfied the conditions:
$i_{k - 1} = j_{k, 0} <j_{k, 1} < \cdots <j_{k,r_k} = i_k$, $1 \leq p_{k,r_k} < \cdots <p_{k,1} \leq n$, and for any $1 \leq l \leq r_k$, $a_{j_{k,l},p_{k,l}} \geq 1$.
$f_{\textbf{(j,p)}}  = \prod_{1 \leq k \leq m, 1 \leq l \leq r_k}f_{k,l}.$

$ f_{k,l} =\left\{\begin{array}{ll}
\frac{v^{2(1 + \sum_{j \geq p_{k,1}}a_{j_{k,0},j})} - v^{2\sum_{j > p_{k,1}}a_{j_{k,0},j}}} {v^2 - 1}\prod\limits_{\xi = j_{k,0} + 1}^{\xi = j_{k,1} - 1}v^{2\sum_{j\geq p_{k,1}}a_{\xi,j}}& \text{if} \ l = 1, p_{k - 1,r_{k - 1}} < p_{k, 1}; \\[.15in]
\frac{v^{2(\sum_{j \geq p_{k,1}}a_{j_{k,0},j})} - v^{2(-1 + \sum_{j > p_{k,1}}a_{j_{k,0},j}})} {v^2 - 1}\prod\limits_{\xi = j_{k,0} + 1}^{\xi = j_{k,1} - 1}v^{2\sum_{j\geq p_{k,1}}a_{\xi,j}}& \text{if} \ l = 1, p_{k - 1,r_{k - 1}} > p_{k, 1}; \\[.15in]
\frac{v^{2(\sum_{j \geq p_{k,1}}a_{j_{k,0},j})} - v^{2( \sum_{j > p_{k,1}}a_{j_{k,0},j}})} {v^2 - 1}\prod\limits_{\xi = j_{k,0} + 1}^{\xi = j_{k,1} - 1}v^{2\sum_{j\geq p_{k,1}}a_{\xi,j}}& \text{if} \ l = 1, p_{k - 1,r_{k - 1}} = p_{k, 1}; \\[.15in]
(v^{2\sum_{j\geq p_{k,l}}a_{j_{k,l - 1},j}} - v^{2\sum_{j> p_{k,l}}a_{j_{k,l - 1},j}}) \prod\limits_{\xi = j_{k,l - 1} + 1}^{k = j_{k,l} - 1}v^{2\sum_{j\geq p_{k,l}}a_{\xi,j}}& \text{if}\  l>1.
 \end{array}   \right.
$

\end{prop}

\begin{proof}
Assume $A' = A + \sum\limits_{\substack{1 \leq k \leq m \\ 1 \leq l \leq r_k}}(E_{j_{k,l - 1},p_{k, l}} - E_{j_{k.l},p_{k,l}})$. Let $f = (V_1 \subset \cdots \subset V_{i_0 - 1} \subset V_{i_0} \subset \cdots \subset V_{i_m - 1} \subset V_{i_m} \subset \cdots \subset V_n), ~f' = (V_1' \subset \cdots \subset V_{i_0 - 1}' \subset V_{i_0}' \subset \cdots \subset V_{i_m - 1}' \subset V_{i_m}' \subset \cdots \subset V_n')$ be such that $(f,f') \in \mathcal{O}_{A'}$. Set $V_0 = V_0' = 0.$  We need to compute how many $f''$ we have such that $(f,~f'') \in \mathcal{O}_B$, $(f'',~f') \in \mathcal{O}_A$.  Assume $ f'' = (V_1'' \subset \cdots \subset V_{i_0 - 1}'' \subset V_{i_0}'' \subset \cdots \subset V_{i_m - 1}'' \subset V_{i_m}'' \subset \cdots \subset V_n'')$.
 For any $i < i_0$ or $i \geq i_m$, $V_i'' = V_i$.
  The proof of this proposition is almost similar to the lemma \ref{lem-y}. The only difference is that when $k > 1,~l = 1 $, how many $V_{j_{k,0}}''$ exist.
 We need to count it in the following three case.

First case.  $ p_{k - 1,r_{k - 1}} < p_{k, 1}$.
Consider the set $Z_{j_{k,0}}$ of all subspace $U_{j_{k,0}}$ of $V$ such that $ V_{j_{k,0} - 1}\subset U_{j_{k,0}} \overset{1}{\subset} V_{j_{k,0}}$, $U_{j_{k,0} - 1} + V_{j_{k,0}} \cap V_{p_{k,1} - 1}' \subset U_{j_{k,0}} $, and $ (U_{j_{k,0} - 1} + V_{j_{k,0}} \cap V_{p_{k,1}}')\cap U_{j_{k,0}} \neq U_{j_{k,0} - 1} + V_{j_{k,0}} \cap V_{p_{k,1}}'$.

 $\sharp Z_{j_{k,0}} = \sharp \{U_{j_{k,0}} | U_{j_{k,0}} \supset V_{j_{k,0} - 1} + V_{j_{k,0}} \cap V_{p_{k,1} - 1}' \} - \sharp \{U_{j_{k,0}} | U_{j_{k,0}} \supset V_{j_{k,0} - 1} + V_{j_{k,0}} \cap V_{p_{k,1}}' \}$

 $ = \frac{v^{2(1 + \sum_{j \geq p_{k,1}}a_{j_{k,0},j})} - v^{2\sum_{j > p_{k,1}}a_{j_{k,0},j}}} {v^2 - 1}.$

 Second case. $ p_{k - 1,r_{k - 1}} = p_{k, 1}$.
 Consider the set $Z_{j_{k,0}}'$ of all subspace $U_{j_{k,0}}'$ of $V$ such that$ V_{j_{k,0} - 1}\subset U_{j_{k,0}}' \overset{1}{\subset} V_{j_{k,0}}$, $U_{j_{k,0} - 1} + V_{j_{k,0}} \cap V_{p_{k,1} - 1}' \subset U_{j_{k,0}}' $, and $ (U_{j_{k,0} - 1} + V_{j_{k,0}} \cap V_{p_{k,1}}')\cap U_{j_{k,0}}' \neq U_{j_{k,0} - 1} + V_{j_{k,0}} \cap V_{p_{k,1}}'$.

 $\sharp Z_{j_{k,0}} = \sharp \{U_{j_{k,0}}' | U_{j_{k,0}}' \supset V_{j_{k,0} - 1} + V_{j_{k,0}} \cap V_{p_{k,1} - 1}' \} - \sharp \{U_{j_{k,0}}' | U_{j_{k,0}}' \supset V_{j_{k,0} - 1} + V_{j_{k,0}} \cap V_{p_{k,1}}' \}$

 $ = \frac{v^{2(\sum_{j \geq p_{k,1}}a_{j_{k,0},j})} - v^{2\sum_{j > p_{k,1}}a_{j_{k,0},j}}} {v^2 - 1}.$

Third case.  $ p_{k - 1,r_{k - 1}} > p_{k, 1}$.
Consider the set $Z_{j_{k,0}}''$ of all subspace $U_{j_{k,0}}''$ of $V$ such that$ V_{j_{k,0} - 1}\subset U_{j_{k,0}}'' \overset{1}{\subset} V_{j_{k,0}}$, $U_{j_{k,0} - 1} + V_{j_{k,0}} \cap V_{p_{k,1} - 1}' \subset U_{j_{k,0}}'' $, and $ (U_{j_{k,0} - 1} + V_{j_{k,0}} \cap V_{p_{k,1}}')\cap U_{j_{k,0}}'' \neq U_{j_{k,0} - 1} + V_{j_{k,0}} \cap V_{p_{k,1}}'$.

 $\sharp Z_{j_{k,0}}'' = \sharp \{U_{j_{k,0}}'' | U_{j_{k,0}}'' \supset V_{j_{k,0} - 1} + V_{j_{k,0}} \cap V_{p_{k,1} - 1}' \} - \sharp \{U_{j_{k,0}}'' | U_{j_{k,0}}'' \supset V_{j_{k,0} - 1} + V_{j_{k,0}} \cap V_{p_{k,1}}' \}$

 $ = \frac{v^{2(\sum_{j \geq p_{k,1}}a_{j_{k,0},j})} - v^{2\sum_{j > p_{k,1}}(-1 + a_{j_{k,0},j}})} {v^2 - 1}.$

 Then the proposition follows.
\end{proof}

We assume that the ground field is an algebraic closure $\overline{\mbb F}_q$ of $\mbb F_q$ when we talk about the dimension of a $G$-orbit or its stabilizer.
Set
\[
d(A)={\rm dim} \ \mcal O_{A}
\quad\mbox{and}\quad
r(A)={\rm dim} \ \mcal O_{B},  \quad \forall A\in \Theta_d ,
\]
where $B=(b_{ij})$ is the diagonal matrix  such that $b_{ii}=\sum_ka_{ik}$.
Denote by ${\rm C}_{G}(V, V')$ the stabilizer of $(V, V')$ in $G$.
From \cite{BLM90}, we have known the following fact.
 If $ A \in \Theta_d$, we have
\begin{equation*}
\begin{split}
{\rm dim}\ {\rm C}_{G}(V, V')
& = \sum_{i\geq k,  j\geq l}a_{ij}a_{kl},
\quad  {\rm if}\ (V, V') \in \mcal O_{A}, \\
{\rm dim}\ \mcal O_{A}
&=\sum_{i<k\ {\rm or}\ j<l}a_{ij}a_{kl}, \\
d(A)-r(A)
&= \sum_{i\geq k,  j<l}a_{ij}a_{kl}.
\end{split}
\end{equation*}
For any $A\in \Theta_{d}, $  let
$[A]=v^{-(d(A)-r(A))}e_{A}$, where $e_{A}$ stand for the characteristic function of $G$-obits associated to $A$.

Define
$$\overline{t}_{ji} = 0, \ \text{if} \ 1 \leq i < j \leq n,$$
$$
\begin{array}{ccc}
  \overline{t}_{ij} & = &  -(v^{-1} - v)\sum\limits_{\substack{ \lambda =(\lambda_1, \lambda_2,\cdots, \lambda_n)\\ s.t \sum_{k =1}^n \lambda_{k}= d - 1}}v^{\lambda_i}[D_{\lambda} + E_{ij}]  \ \text{if} \  1\leq i < j \leq n,\\ [.4 in]
  \overline{t}_{ii} & = & \sum\limits_{\substack{ \lambda =(\lambda_1, \lambda_2,\cdots, \lambda_n) \\ s.t \sum_{k =1}^n \lambda_{k}= d}}v^{\lambda_i}[D_{\lambda}] \ \text{if} \ 1 \leq i \leq n.\\
\end{array}
$$
\begin{prop}
The functions $t_{ij},~\overline{t}_{ij}$ in $\mcal S$, for any $i,~j\in [1,~n]$,  satisfy the following relations.
\begin{eqnarray*}
  (R1)&&v^{-\delta_{ij}} \overline{t}_{ia}\overline{t}_{jb} - v^{-\delta_{ab}} \overline{t}_{jb}\overline{t}_{ia} = (v^{-1}- v)(\delta_{b < a} - \delta_{i < j})\overline{t}_{ja}\overline{t}_{ib},\\
  (R2)&&\prod_{i=1}^n \overline{t}_{ii}=v^{d},\\
  (R3)&&  \prod_{l=0}^d(\overline{t}_{ii} - v^{l})=0, ~\forall ~i \in [1,~n].\\
  (R4)&&  {\overline{t}_{ij}}^{d+1} = 0,~\rm{if} \  i \neq j.
\end{eqnarray*}
\end{prop}

\begin{proof}
We show the identity R1.

We will show it in the following several cases. By the first part of  lemma \ref{lem-y}, the following identities can be obtained by directly computing.

First case.  $1 \leq i < j < b < a \leq n$,
$$
\begin{array}{ccc}
  \overline{t}_{ia}\overline{t}_{jb}& = & (v - v^{-1})^{2}\sum\limits_{ \lambda}v^{\lambda_i}[D_{\lambda} + E_{ia}] \sum\limits_{ \lambda'}v^{\lambda_j}[D_{\lambda} + E_{jb}]\\[.15 in]
   & = & (v - v^{-1})^{2}\sum\limits_{\lambda}v^{- \sum_{k = i + 1}^{a - 1}\lambda_k}e_{D_{\lambda} + E_{ia}}\sum\limits_{\lambda'}v^{- \sum_{k = j + 1}^{b - 1}\lambda_k}e_{D_{\lambda'} + E_{ia}}  \\ [.15 in]
   & = &(v - v^{-1})^{2} \sum\limits_{\substack{ \lambda =(\lambda_1, \lambda_2,\cdots, \lambda_n)\\ s.t \sum\limits_{k =1}^n \lambda_{k}= d - 2}} v^{- \sum_{k = i + 1}^{a - 1}\lambda_{k} - \sum_{k = j + 1}^{b - 1}\lambda_{k} - 1}e_{D_{\lambda} + E_{ia} + E_{jb}}. \\
\end{array}
$$
$$
\begin{array}{ccc}
  \overline{t}_{jb}\overline{t}_{ia}& = & (v - v^{-1})^{2}\sum\limits_{ \lambda'}v^{\lambda_j'}[D_{\lambda} + E_{jb}]\sum\limits_{ \lambda}v^{\lambda_i}[D_{\lambda} + E_{ia}] \\ [.15 in]
   & = & (v - v^{-1})^{2}\sum\limits_{\lambda'}v^{- \sum_{k = j + 1}^{b - 1}\lambda_k}e_{D_{\lambda'} + E_{ia}}\sum\limits_{\lambda}v^{- \sum_{k = i + 1}^{a - 1}\lambda_k}e_{D_{\lambda} + E_{ia}}  \\ [.15 in]
   & = &(v - v^{-1})^{2} \sum\limits_{\substack{ \lambda =(\lambda_1, \lambda_2,\cdots, \lambda_n) \\s.t \sum_{k =1}^n \lambda_{k}= d - 2}}v^{- \sum_{k = i + 1}^{a - 1}\lambda_{k} - \sum_{k = j + 1}^{b - 1}\lambda_{k} - 1}e_{D_{\lambda} + E_{ia} + E_{jb}}. \\
\end{array}
$$
Hence
$$ \overline{t}_{ia}\overline{t}_{jb} -  \overline{t}_{jb}\overline{t}_{ia} = 0.$$

Second case.  $i < j \ \text{and} \ a \leq b$,

$ \overline{t}_{ia}\overline{t}_{jb} \\
 =  (v^{-1} - v)^{2}\sum\limits_{ \lambda}v^{\lambda_i}[D_{\lambda} + E_{ia}] \sum\limits_{ \lambda'}v^{\lambda_j}[D_{\lambda} + E_{jb}]\\ [.15 in]
  = (v^{-1} - v)^{2}\sum\limits_{\lambda}v^{- \sum_{k = i + 1}^{a - 1}\lambda_k}e_{D_{\lambda} + E_{ia}}\sum\limits_{\lambda'}v^{- \sum_{k = j + 1}^{b - 1}\lambda_k}e_{D_{\lambda'} + E_{ia}}  \\ [.15 in]
=\left\{\begin{array}{ll}
(v^{-1} - v)^{2} \sum\limits_{ \substack{\lambda =(\lambda_1, \lambda_2,\cdots, \lambda_n) \\s.t \sum_{k =1}^n \lambda_{k}= d - 2}} v^{- \sum_{k = i + 1}^{a - 1}\lambda_{k} - \sum_{k = j + 1}^{b - 1}\lambda_{k} + 1}e_{D_{\lambda} + E_{ia} + E_{jb}}
& \text{if} \ i < j < a = b; \\[.4in]
(v^{-1} - v)^{2} (\sum\limits_{ \substack{\lambda =(\lambda_1, \lambda_2,\cdots, \lambda_n) \\s.t \sum_{k =1}^n \lambda_{k}= d - 2}} v^{- \sum_{k = i + 1}^{a - 1}\lambda_{k} - \sum_{k = j + 1}^{b - 1}\lambda_{k}}e_{D_{\lambda} + E_{ia} + E_{jb}} \\ +  \sum\limits_{ \substack{\lambda =(\lambda_1, \lambda_2,\cdots, \lambda_n) \\s.t \sum_{k =1}^n \lambda_{k}= d - 2}} v^{- \sum_{k = i + 1}^{a - 1}\lambda_{k} - \sum_{k = j + 1}^{b - 1}\lambda_{k} - 2}(v^{2} - 1)e_{D_{\lambda} + E_{ib} + E_{ja}})
& \text{if} \ i < j < a < b ; \\[.4in]
(v^{-1} - v)^{2} (\sum\limits_{ \substack{\lambda =(\lambda_1, \lambda_2,\cdots, \lambda_n) \\s.t \sum_{k =1}^n \lambda_{k}= d - 2}} v^{- \sum_{k = i + 1}^{a - 1}\lambda_{k} - \sum_{k = j + 1}^{b - 1}\lambda_{k}}e_{D_{\lambda} + E_{ia} + E_{jb}} \\ +  \sum\limits_{ \substack{\lambda =(\lambda_1, \lambda_2,\cdots, \lambda_n) \\s.t \sum_{k =1}^n \lambda_{k}= d - 1}} v^{- \sum_{k = i + 1}^{a - 1}\lambda_{k} - \sum_{k = j + 1}^{b - 1}\lambda_{k}}e_{D_{\lambda} + E_{ib}})
& \text{if} \ i < j = a < b; \\[.4in]
(v^{-1} - v)^{2} \sum\limits_{ \substack{\lambda =(\lambda_1, \lambda_2,\cdots, \lambda_n) \\s.t \sum_{k =1}^n \lambda_{k}= d - 2}} v^{- \sum_{k = i + 1}^{a - 1}\lambda_{k} - \sum_{k = j + 1}^{b - 1}\lambda_{k} }e_{D_{\lambda} + E_{ia} + E_{jb}}
& \text{if}\  i < a < j < b.
 \end{array}   \right.
$

$  \overline{t}_{jb}\overline{t}_{ia} \\ [.15 in]
 =  (v^{-1} - v)^{2}\sum_{ \lambda'}v^{\lambda_j}[D_{\lambda} + E_{jb}]\sum_{ \lambda}v^{\lambda_i}[D_{\lambda} + E_{ia}] \\ [.15 in]
    =  (v^{-1} - v)^{2}\sum_{\lambda'}v^{- \sum_{k = j + 1}^{b - 1}\lambda_k}e_{D_{\lambda'} + E_{ia}}\sum_{\lambda}v^{- \sum_{k = i + 1}^{a - 1}\lambda_k}e_{D_{\lambda} + E_{ia}}  \\ [.15 in]
    =(v^{-1} - v)^{2} \sum\limits_{ \substack{\lambda =(\lambda_1, \lambda_2,\cdots, \lambda_n) \\s.t \sum_{k =1}^n \lambda_{k}= d - 2}} v^{- \sum_{k = i + 1}^{a - 1}\lambda_{k} - \sum_{k = j + 1}^{b - 1}\lambda_{k} }e_{D_{\lambda} + E_{ia} + E_{jb}}.
$
\\ [.2 in]
$ \overline{t}_{ja}\overline{t}_{ib}
=\left\{\begin{array}{ll}
(v^{-1} - v)^{2} \sum\limits_{ \substack{\lambda =(\lambda_1, \lambda_2,\cdots, \lambda_n) \\s.t \sum_{k =1}^n \lambda_{k}= d - 2}} v^{- \sum_{k = i + 1}^{a - 1}\lambda_{k} - \sum_{k = j + 1}^{b - 1}\lambda_{k}}e_{D_{\lambda} + E_{ib} + E_{ja}}
& \text{if} \ i < j < a = b; \\[.15in]
(v^{-1} - v)^{2} \sum_{ \substack{\lambda =(\lambda_1, \lambda_2,\cdots, \lambda_n)\\ s.t \sum_{k =1}^n \lambda_{k}= d - 2}} v^{- \sum_{k = i + 1}^{a - 1}\lambda_{k} - \sum_{k = j + 1}^{b - 1}\lambda_{k} - 1}e_{D_{\lambda} + E_{ib} + E_{ja}}
& \text{if} \ i < j < a < b ; \\[.15in]
-(v^{-1} - v)\sum\limits_{ \substack{\lambda =(\lambda_1, \lambda_2,\cdots, \lambda_n) \\s.t \sum_{k =1}^n \lambda_{k}= d - 1}} v^{- \sum_{k = i + 1}^{a - 1}\lambda_{k} - \sum_{k = j + 1}^{b - 1}\lambda_{k}}e_{D_{\lambda} + E_{ib}}
& \text{if} \ i < j = a < b; \\[.15in]
0
& \text{if}\  i < a < j < b.
 \end{array}   \right.
$

Thus,
$$ \overline{t}_{ia}\overline{t}_{jb} - v^{-\delta_{ab}} \overline{t}_{jb}\overline{t}_{ia} = -(v^{-1}- v)\overline{t}_{ja}\overline{t}_{ib}.$$

Third case. $b < a\ \text{and} \  i \geq j$,

$ \overline{t}_{ia}\overline{t}_{jb} \\ [.15 in]
  = (v^{-1} - v)^{2}\sum_{\lambda}v^{- \sum_{k = i + 1}^{a - 1}\lambda_k}e_{D_{\lambda} + E_{ia}}\sum_{\lambda'}v^{- \sum_{k = j + 1}^{b - 1}\lambda_k}e_{D_{\lambda'} + E_{ia}}  \\ [.15 in]
  = (v^{-1} - v)^{2} \sum\limits_{ \substack{\lambda =(\lambda_1, \lambda_2,\cdots, \lambda_n) \\s.t \sum_{k =1}^n \lambda_{k}= d - 2}} v^{- \sum_{k = i + 1}^{a - 1}\lambda_{k} - \sum_{k = j + 1}^{b - 1}\lambda_{k} }e_{D_{\lambda} + E_{ia} + E_{jb}}.
$

$  \overline{t}_{jb}\overline{t}_{ia}
 =  (v^{-1} - v)^{2}\sum_{ \lambda'}v^{\lambda_j}[D_{\lambda} + E_{jb}]\sum_{ \lambda}v^{\lambda_i}[D_{\lambda} + E_{ia}] \\ [.15 in]
    =  (v^{-1} - v)^{2}\sum_{\lambda'}v^{- \sum_{k = j + 1}^{b - 1}\lambda_k}e_{D_{\lambda'} + E_{jb}}\sum_{\lambda}v^{- \sum_{k = i + 1}^{a - 1}\lambda_k}e_{D_{\lambda} + E_{ia}}  \\ [.15 in]
=\left\{\begin{array}{ll}
(v^{-1} - v)^{2} \sum\limits_{ \substack{\lambda =(\lambda_1, \lambda_2,\cdots, \lambda_n) \\s.t \sum_{k =1}^n \lambda_{k}= d - 2}} v^{- \sum_{k = i + 1}^{a - 1}\lambda_{k} - \sum_{k = j + 1}^{b - 1}\lambda_{k} }e_{D_{\lambda} + E_{jb} + E_{ia}}
& \text{if} \ j < b < i < a; \\[.4in]
(v^{-1} - v)^{2} (\sum\limits_{ \substack{\lambda =(\lambda_1, \lambda_2,\cdots, \lambda_n) \\s.t \sum_{k =1}^n \lambda_{k}= d - 2}} v^{- \sum_{k = i + 1}^{a - 1}\lambda_{k} - \sum_{k = j + 1}^{b - 1}\lambda_{k}}e_{D_{\lambda} + E_{ia} + E_{jb}} \\ +  \sum\limits_{ \substack{\lambda =(\lambda_1, \lambda_2,\cdots, \lambda_n) \\s.t \sum_{k =1}^n \lambda_{k}= d - 1}} v^{- \sum_{k = i + 1}^{a - 1}\lambda_{k} - \sum_{k = j + 1}^{b - 1}\lambda_{k}}e_{D_{\lambda} + E_{ja}})
& \text{if} \ j < b = i < a ; \\[.4in]
(v^{-1} - v)^{2} (\sum\limits_{ \substack{\lambda =(\lambda_1, \lambda_2,\cdots, \lambda_n) \\s.t \sum_{k =1}^n \lambda_{k}= d - 2}} v^{- \sum_{k = i + 1}^{a - 1}\lambda_{k} - \sum_{k = j + 1}^{b - 1}\lambda_{k}}e_{D_{\lambda} + E_{ia} + E_{jb}} \\ +  \sum\limits_{ \substack{\lambda =(\lambda_1, \lambda_2,\cdots, \lambda_n) \\s.t \sum_{k =1}^n \lambda_{k}= d - 2}} v^{- \sum_{k = i + 1}^{a - 1}\lambda_{k} - \sum_{k = j + 1}^{b - 1}\lambda_{k} - 2}(v^{2} - 1)e_{D_{\lambda} + E_{ib} + E_{ja}})
& \text{if} \ j < i < b < a; \\[.4in]
(v^{-1} - v)^{2} \sum\limits_{ \substack{\lambda =(\lambda_1, \lambda_2,\cdots, \lambda_n) \\s.t \sum_{k =1}^n \lambda_{k}= d - 2}} v^{- \sum_{k = i + 1}^{a - 1}\lambda_{k} - \sum_{k = j + 1}^{b - 1}\lambda_{k} + 1}e_{D_{\lambda} + E_{jb} + E_{ia}}
& \text{if}\  j = i < b < a.
 \end{array}   \right.
$

$  \overline{t}_{ja}\overline{t}_{ib}
=\left\{\begin{array}{ll}
0
& \text{if} \ j < b < i < a; \\[.15in]
-(v^{-1} - v)\sum\limits_{ \substack{\lambda =(\lambda_1, \lambda_2,\cdots, \lambda_n) \\s.t \sum_{k =1}^n \lambda_{k}= d - 1}} v^{- \sum_{k = i + 1}^{a - 1}\lambda_{k} - \sum_{k = j + 1}^{b - 1}\lambda_{k}}e_{D_{\lambda} + E_{ja}}
& \text{if} \ j < b = i < a ; \\[.15in]
(v^{-1} - v)^{2}\sum\limits_{ \substack{\lambda =(\lambda_1, \lambda_2,\cdots, \lambda_n) \\s.t \sum_{k =1}^n \lambda_{k}= d - 2}} v^{- \sum_{k = i + 1}^{a - 1}\lambda_{k} - \sum_{k = j + 1}^{b - 1}\lambda_{k} - 1}e_{D_{\lambda} + E_{ib} + E_{ja}}
& \text{if} \ j < i < b < a; \\[.15in]
(v^{-1} - v)^{2} \sum\limits_{ \substack{\lambda =(\lambda_1, \lambda_2,\cdots, \lambda_n) \\s.t \sum_{k =1}^n \lambda_{k}= d - 2}} v^{- \sum_{k = i + 1}^{a - 1}\lambda_{k} - \sum_{k = j + 1}^{b - 1}\lambda_{k} }e_{D_{\lambda} + E_{ja} + E_{ib}}
& \text{if}\  j = i < b < a.
 \end{array}   \right.
$

Therefore,
$$v^{-\delta_{ij}} \overline{t}_{ia}\overline{t}_{jb} -  \overline{t}_{jb}\overline{t}_{ia} = (v^{-1}- v)\overline{t}_{ja}\overline{t}_{ib}.$$

Forth case. $j \leq i < a \leq b$.  it is easy to know that when $j = i < a =b$, the identity R1 is equal. From the first case, when$j < i  < a < b$, the identity R1 is also equal. From the second and the third case. When $j < i  < a = b$ and $j = i  < a < b$, the identity is right.

Thus, the identity R1 is equal.

 The other identities can be shown similarly to Proposition 4.1.1 in \cite{MZZ17}.
\end{proof}

Recall the partial order $``\leq"$ on $\Theta_d$  by $A\leq B$ if $\mcal O_{A} \subset \overline{\mcal O}_{B}$\cite{BLM90}.
For any $A=(a_{ij})$ and $B=(b_{ij})$ in $\Theta_{\mbf d}$,  we say that $A \preceq B$ if and only if
the following two conditions hold.
\begin{align}\label{partial-order}
 \sum_{r\leq i,  s\geq j} a_{rs}  & \leq \sum_{r\leq i,  s\geq j} b_{rs},  \quad \forall i<j. \\
  \sum_{r\geq i,  s\leq j} a_{rs}  & \leq \sum_{r\geq i,  s\leq j} b_{rs},  \quad \forall i>j.
\end{align}
The  relation $``\preceq" $ defines a second partial order on  $\Theta_d$.

\begin{thm}
For any $A=(a_{ij})\in \Theta_{\mbf d}$.  The following identity holds in $\mcal S$
$$\prod_{1\leq j  < i \leq n} e_{D_{ij} + a_{ij}E_{ij}} *\prod_{1\leq i < j \leq n} e_{D_{ij} + a_{ij}E_{ij}}   =  \chi_A e_A + {\rm lower\ terms}, $$
where $\chi_A \in \mathcal A \setminus\{0\}$.
The factors in the first product are taken in the following order: $(i, ~j)$ comes before $(i', ~ j')$ if either $j < j'$ or $j = j',~ i< i'$.
The factors in the second product are taken in the following order: $(i,~j)$ comes before $(i',~j')$ if either $j > j'$ or $j = j',~ i> i'$.
The matrices $D_{i, j}$ are diagonal with entries in $\mathbb{N}$,which are uniquely determined.
\end{thm}

\begin{proof}
Assume n = 3 so that
$$
A=\left(
  \begin{array}{ccc}
    a_{11} & a_{12} & a_{13} \\
    a_{21} & a_{22} & a_{23} \\
    a_{31} & a_{32} & a_{33} \\
  \end{array}
\right)
$$
We define
$$
A_1=\left(
  \begin{array}{ccc}
    a_{11}+a_{21}+a_{31} & a_{12} & 0 \\
    0 & a_{22}+a_{32} & 0 \\
    0 & 0 & a_{13}+a_{23}+a_{33} \\
  \end{array}
\right)
$$

$$
A_2=\left(
  \begin{array}{ccc}
    a_{11}+a_{21}+a_{31}+a_{12} & 0 & a_{13} \\
    0 & a_{22}+a_{32} & 0 \\
    0 & 0 & a_{23}+a_{33} \\
  \end{array}
\right)
$$

$$
A_3=\left(
  \begin{array}{ccc}
    a_{11}+a_{21}+a_{31}+a_{12}+a_{13} & 0 & 0\\
    0 & a_{22}+a_{32} & a_{23} \\
    0 & 0 & a_{33} \\
  \end{array}
\right)
$$

$$
A_4=\left(
  \begin{array}{ccc}
    a_{11}+a_{21}+a_{31}+a_{12}+a_{13} & 0 & 0\\
    0 & a_{22}+a_{23} & 0 \\
    0 & a_{32} & a_{33} \\
  \end{array}
\right)
$$

$$
A_5=\left(
  \begin{array}{ccc}
    a_{11}+a_{21}+a_{12}+a_{13} & 0 & 0\\
    0 & a_{22}+a_{23} & 0 \\
    a_{31} & 0 &a_{32}+a_{33} \\
  \end{array}
\right)
$$

$$
A_6=\left(
  \begin{array}{ccc}
    a_{11}+a_{12}+a_{13} & 0 & 0\\
    a_{21} & a_{22}+a_{23} & 0 \\
     0 & 0 &a_{31}+a_{32}+a_{33} \\
  \end{array}
\right)
$$

By the Lemma \ref{lem-y}, we have

$$e_{A_{6}}*e_{A_{5}}* \cdots *e_{A_{1}}  = \chi_{A}e_{A} + \mbox{lower \ terms}.$$

Similarly,  we can prove the proposition for the general case.

\end{proof}

\subsection{Stabilization}

Let $I$ be the identity matrix.  We set
$
{}_pA=A + pI
$.
Let $\widetilde{\Theta}$  be the set of all $n\times n$ matrix with integer entries such that the entries off diagonal are $\geq 0$.

Let
$$\mcal K= \mbox{span}_{\mcal A} \{ [A]~ |~ A\in \wt{\Theta} \},
$$
where the notation $[A]$ is a formal  symbol.
Let $v'$ be a independent indeterminates,  and  $\mfk R$  be the ring $\mathbb{Q}(v)[v']$.

From \cite{BLM90}, we have known the following results.
\begin{prop}\label{prop1}
Suppose that $A_1, ~ A_2,~ \cdots, ~ A_r \ (r\geq 2)$ are matrices in $\wt{\Theta}$
such that ${\rm co}(A_i)={\rm ro}(A_{i+1})$  for $1\leq i \leq r-1$.
There exist $Z_1, ~ \cdots,~  Z_m\in \wt{\Theta}$,  $G_j(v, v')\in \mfk R$ and $p_0\in \mbb N$ such that in $\mcal S_d$ for some $d$,   we have
$$[{}_p A_1] * [{}_pA_2] * \cdots *[{}_p A_r]=\sum_{j=1}^mG_j(v, v^{-p})[{}_p Z_j], \quad
\forall  p\geq p_0. $$
\end{prop}
By specialization $v'$ at $v'=1$,  there is a unique associative $\mcal A$-algebra structure on $\mcal K$ without unit,   where
 the product is given by
 $$[A_1] \cdot [\A_2]\cdot \dots \cdot  [\A_r] =\sum_{j=1}^m G_j(v, 1)[Z_j]$$
 if $A_1, \cdots,  A_r$ are as in Proposition \ref{prop1}.

Let $\hat{\mcal K}$ be the $\mbb Q(v)$-vector space of all formal sum
$\sum\limits_{A\in \tilde{\Theta}}\xi_{A} [A]$ with $\xi_{A}\in \mbb Q(v)$ and  a locally finite property,  that is,
for any ${\mbf t}\in \mbb Z^n$,  the sets $\{A\in \tilde{\Theta}~|~{\rm ro}(A)={\mbf t}, ~ \xi_{A} \neq 0\}$
and
$\{A\in \widetilde{\Theta}~ |~ {\rm co}(A)={\mbf t},~  \xi_{A} \neq 0\}$ are finite.
The space $\hat{\mcal K}$ becomes an  associative algebra over $\mbb Q(v)$
which equipped  with  the following multiplication:
$$
\sum_{A\in \wt{\Theta}} \xi_{A} [A]   \cdot \sum_{B \in \wt{\Theta}} \xi_{B} [B]
=\sum_{A,  B \in \wt{\Theta}} \xi_{A} \xi_{B} [A] \cdot [B],
$$
where the product $[A] \cdot [B]$ is taken  in $\mcal K$.

For any nonzero  matrix $A \in \wt{\Theta}$,
let $\hat{A}$ be the matrix obtained
by replacing diagonal entries of $A$ by zeroes.
We set
$
\Theta^{0}= \{ \hat{A} | A\in \wt{\Theta} \}.
$

For any $\hat{\A}$ in $\Theta^{0}$ and ${\mbf j}=(j_1, \cdots,  j_n)\in \mbb Z^n$,  we define
\begin{equation} \label{1wtA}
\hat{\A} ({\mbf j})=\sum_{\lambda}v^{\lambda_1j_1+\cdots+\lambda_{n}j_{n}}[\hat{\A} + D_{\lambda} \}],\quad
\end{equation}
where the  sum runs through all $\lambda=(\lambda_i)\in \mbb Z^n$ such that
$\hat{\A} + D_{\lambda} \in \wt{\Theta}$, where $D_{\lambda}$ is the diagonal matrices with diagonal entries $(\lambda_i). $

Define
$$\overline{t}_{ji} = 0, \ \text{if} \ 1 \leq i < j \leq n,$$
$$
\begin{array}{ccc}
  \overline{t}_{ij} & =&   -(v^{-1} - v)E_{ij}(\underline i ) \ (1\leq i < j \leq n),\\
  \overline{t}_{ii} & =& 0(\underline i)\ (1 \leq i \leq n),\\
\end{array}
$$
where $\underline i \in \mbb N^n$ is the vector whose $i$-th entry is 1 and 0 elsewhere.

Let $\mcal U$ be the subalgebra of $\hat{\mcal K}$ generated by $ \overline{t}_{ij}, \overline{t}_{ii} $ for all $i, j\in [1, n]$ and $\mbf j\in \mbb Z^n$.

\begin{prop}
The generators $t_{ij},~\overline{t}_{ij}$ in $\mcal U$, for any $i,~j\in [1,~n]$, satisfy the following relations.
\begin{eqnarray*}
  v^{-\delta_{ij}} \overline{t}_{ia}\overline{t}_{jb} - v^{-\delta_{ab}} \overline{t}_{jb}\overline{t}_{ia} = (v^{-1}- v)(\delta_{b < a} - \delta_{i < j})\overline{t}_{ja}\overline{t}_{ib},\\
\end{eqnarray*}
\end{prop}

\begin{proof}
 The following identities can be obtained by directly computing by the first part of  Lemma \ref{lem-y} and Proposition \ref{prop1}.

First case.  $1 \leq i < j < b < a \leq n$,
$$
\begin{array}{ccc}
  \overline{t}_{ia}\overline{t}_{jb}& = & (v - v^{-1})^{2}E_{ia}(\underline i)E_{jb}(\underline j)\\ [.15 in]
  &=&(v - v^{-1})^{2}\sum\limits_{ \lambda}v^{\lambda_i}[D_{\lambda} + E_{ia}] \sum\limits_{ \lambda'}v^{\lambda_j'}[D_{\lambda} + E_{jb}]\\ [.15 in]
   & = &(v - v^{-1})^{2} \sum_{ \lambda} v^{\lambda_i + \lambda_j}[D_{\lambda} + E_{ia} + E_{jb}] \\ [.15 in]
   & = &(v - v^{-1})^{2}(E_{ia} + E_{jb})(\underline i + \underline j). \\
\end{array}
$$
$$
\begin{array}{ccc}
  \overline{t}_{jb}\overline{t}_{ia}& = & (v - v^{-1})^{2}E_{jb}(\underline j)E_{ia}(\underline i)\\ [.15 in]
  & = & (v - v^{-1})^{2}\sum_{ \lambda'}v^{\lambda_j'}[D_{\lambda} + E_{jb}]\sum_{ \lambda}v^{\lambda_i}[D_{\lambda} + E_{ia}] \\ [.15 in]
   & = &(v - v^{-1})^{2} \sum_{ \lambda} v^{\lambda_i + \lambda_j}[D_{\lambda} + E_{ia} + E_{jb}] \\ [.15 in]
   & = &(v - v^{-1})^{2}(E_{ia} + E_{jb})(\underline i + \underline j). \\ [.15 in]
\end{array}
$$
Hence
$$ \overline{t}_{ia}\overline{t}_{jb} -  \overline{t}_{jb}\overline{t}_{ia} = 0.$$

Second case.  $i < j \ \text{and} \  a \leq b$,

$ \overline{t}_{ia}\overline{t}_{jb} \\ [.15 in]
 =  (v^{-1} - v)^{2}\sum_{ \lambda}v^{\lambda_i}[D_{\lambda} + E_{ia}] \sum_{ \lambda'}v^{\lambda_j'}[D_{\lambda} + E_{jb}]\\ [.15 in]
=\left\{\begin{array}{ll}
(v^{-1} - v)^{2} v(E_{ia}+E_{jb})(\underline i + \underline j),
& \text{if} \ i < j < a = b; \\[.15in]
(v^{-1} - v)^{2} ( (E_{ia}+E_{jb})(\underline i + \underline j) +   (v - v^{-1})(E_{ib} + E_{ja})(\underline i + \underline j))
& \text{if} \ i < j < a < b ; \\[.15in]
(v^{-1} - v)^{2} ((E_{ia}+E_{jb})(\underline i + \underline j)  + (E_{ib})(\underline i + \underline j) )
& \text{if} \ i < j = a < b; \\[.15in]
(v^{-1} - v)^{2} (E_{ia}+E_{jb})(\underline i + \underline j).
& \text{if}\  i < a < j < b.
 \end{array}   \right.
$

$  \overline{t}_{jb}\overline{t}_{ia} \\ [.15 in]
 =  (v^{-1} - v)^{2}\sum_{ \lambda'}v^{\lambda_j'}[D_{\lambda} + E_{jb}]\sum_{ \lambda}v^{\lambda_i}[D_{\lambda} + E_{ia}] \\ [.15 in]
  =  (v^{-1} - v)^{2}(E_{ia}+E_{jb})(\underline i + \underline j).
$

$ \overline{t}_{ja}\overline{t}_{ib}
=\left\{\begin{array}{ll}
(v^{-1} - v)^{2}(E_{ia}+E_{jb})(\underline i + \underline j).
& \text{if} \ i < j < a = b; \\[.15in]
(v^{-1} - v)^{2}  (E_{ib} + E_{ja})(\underline i + \underline j)
& \text{if} \ i < j < a < b ; \\[.15in]
-(v^{-1} - v) E_{ib}(\underline i + \underline j)
& \text{if} \ i < j = a < b; \\[.15in]
0
& \text{if}\  i < a < j < b.
 \end{array}   \right.
$

Thus,
$$ \overline{t}_{ia}\overline{t}_{jb} - v^{-\delta_{ab}} \overline{t}_{jb}\overline{t}_{ia} = -(v^{-1}- v)\overline{t}_{ja}\overline{t}_{ib}.$$

Third case.  $b < a \ \text{and} \ i \geq j$,

$ \overline{t}_{ia}\overline{t}_{jb} \\ [.15 in]
   =  (v^{-1} - v)^{2}\sum_{ \lambda}v^{\lambda_i}[D_{\lambda} + E_{ia}] \sum_{ \lambda'}v^{\lambda_j'}[D_{\lambda} + E_{jb}]\\ [.15 in]
  = (v^{-1} - v)^{2} v^{\delta_{ij}}(E_{ia}+E_{jb})(\underline i + \underline j)
$

$  \overline{t}_{jb}\overline{t}_{ia} \\ [.15 in]
 =  (v^{-1} - v)^{2}\sum_{ \lambda'}v^{\lambda_j'}[D_{\lambda} + E_{jb}]\sum_{ \lambda}v^{\lambda_i}[D_{\lambda} + E_{ia}] \\ [.15 in]
=\left\{\begin{array}{ll}
(v^{-1} - v)^{2}(E_{ia}+E_{jb})(\underline i + \underline j),
& \text{if} \ j < b < i < a; \\[.15in]
(v^{-1} - v)^{2} ((E_{ia}+E_{jb})(\underline i + \underline j) +   E_{ja}(\underline i + \underline j))
& \text{if} \ j < b = i < a ; \\[.15in]
(v^{-1} - v)^{2} ( (E_{ia}+E_{jb})(\underline i + \underline j) +  (v - v^{-1})( E_{ib} + E_{ja})(\underline i + \underline j))
& \text{if} \ j < i < b < a; \\[.15in]
(v^{-1} - v)^{2}  v^2(E_{jb} + E_{ia})(\underline i + \underline j)
& \text{if}\  j = i < b < a.
 \end{array}   \right.
$

$  \overline{t}_{ja}\overline{t}_{ib} \\
=\left\{\begin{array}{ll}
0
& \text{if} \ j < b < i < a; \\[.15in]
-(v^{-1} - v) E_{ja}(\underline i + \underline j)
& \text{if} \ j < b = i < a ; \\[.15in]
(v^{-1} - v)^{2}(E_{ib} + E_{ja})(\underline i + \underline j)
& \text{if} \ j < i < b < a; \\[.15in]
(v^{-1} - v)^{2} v(E_{ja} + E_{ib})(\underline i + \underline j)
& \text{if}\  j = i < b < a.
 \end{array}   \right.
$

Therefore
$$v^{-\delta_{ij}} \overline{t}_{ia}\overline{t}_{jb} -  \overline{t}_{jb}\overline{t}_{ia} = (v^{-1}- v)\overline{t}_{ja}\overline{t}_{ib}.$$

Forth case. $j \leq i < a \leq b$.  it is easy to know that when $j = i < a =b$, the identity is equal. From the first case, when$j < i  < a < b$, the identity is also equal.From the second and the third case. When $j < i  < a = b$ and $j = i  < a < b$, the identity is right.
Thus, The proposition follows.

\end{proof}

The previous proposition shows that $\U$ satisfied the defining relations of the positive part of $U_v(gl_n)$ with respect to the RTT relations. In fact , $\U \cong U_v(gl_n)^{+}$. That is, we give the realization of the  $U_v(gl_n)^{+}$ with respect to the RTT relations.

\bigskip

\noindent{{\bf Acknowledgements:} This work is supported by NSFC 11571119 and NSFC 11475178.}


\begin{thebibliography}{99999}\frenchspacing


\bibitem{BLM90} A. Beilinson, G. Lusztig, R. McPherson,
          {\em A geometric setting for the quantum deformation of $GL_n$}, Duke Math. J., {\bf 61} (1990), 655-677.

\bibitem{BKLW13}  H.  Bao,  J.  Kujawa,  Y.  Li,   W.  Wang,  {\em Geometric Schur duality of classical type,  with Appendix A by H.  Bao,  Y.  Li,  and W.  Wang},   arXiv:1404. 4000.
\bibitem{FL14} Z.  Fan,  Y.  Li,  {\em Geometry Schur duality of classical type, II},  arXiv:1408. 6740v1.

\bibitem{F12} Q. Fu,
                {\em BLM realization for ${\mathcal U}_{\mbb Z} (\widehat{\mathfrak{gl}}_n)$}, arXiv:1204.3142.
\bibitem{DG14} J. Du and H. Gu,
{\em A realisation of the quantum general linear superalgebra}, J. Algebra 401 (2014) 60-99.



\bibitem{Lu93} G. Lusztig, {\em Introduction to Quantum groups}, Modern Birkh{\"a}user Classics,
Reprint of the 1993 Edition, Birkh\"{a}user, Boston, 2010.



\bibitem{SV00}  O. Schiffmann and E. Vasserot,
         {\em Geometric construction of the global base of the quantum modified algebra of $\widehat{ \mathfrak{gl}}_n$},
         Transform. Groups {\bf 5} (2000), 351--360.


\bibitem{BK} J. Brundan and A. Kleshchev, {\em Parabolic presentations of the Yangian $Y(\mathfrak{gl}_n)$}, Comm. Math. Phys. 254 (2005), 191-220.

\bibitem{DF} J. Ding and I. B. Frenkel, {\em Isomorphism of two realizations of quantum affine algebra $U_q (\widehat{\mathfrak{gl}(n)})$}, Comm. Math. Phys. 156 (1993), 277--300.

\bibitem{D1} V. Drinfeld, {\em Hopf algebras and the quantum Yang-Baxter
equation}, Soviet Math. Dokl. 32 (1985), 254--258.

\bibitem{D2} V. Drinfeld, {\em A new realization of Yangians and quantized affine
algebras}, Soviet Math. Dokl. 36 (1988), 212--216.

\bibitem{D3} V. Drinfeld, {\em Quantum Group}, Proc. ICM, Vol. 1, 2 (Berkeley, Calif., 1986), 798--820, Amer. Math. Soc., Providence, RI, 1987.

\bibitem{FRT} L. Faddeev, N. Reshetikhin and L. Takhtadzhyan, {\em Quantization
of Lie groups and Lie algebras}, Leningrad Math. J. 1 (1990),
193--225.

\bibitem{FRT2} L. Faddeev, N. Reshetikhin and L. Takhtajan, {\em Quantization
of Lie groups and Lie algebras}, in: Yang-Baxter equations and integrable systems, Advanced Series in Mathematical Physics,
10, World Scientific, Singapore, pp. 299--309, 1989.


\bibitem{J} M. Jimbo, {\em A q-difference analogue of $U(\mathfrak g)$ and the Yang-Baxter equation}, Lett. Math. Phys. 10 (1985), 63--69.
\bibitem{JL2} N. Jing and M. Liu, {\em R-matrix realization of two-parameter quantum group $U_{r,s}(\mathfrak{gl}_n)$},
Comm. Math. Stat. 2 (2014),
211--230.
363 (2011), 3769--3797.
\bibitem{MZZ17}H. Ma, Z. Lin, Z. Zheng. {\em Geometric Schur-Weyl Duality of two parameter quantum group of type A }. arXiv:1701.06114


\end{thebibliography}
\end{document}